\documentclass[12pt,a4paper]{amsart}
\usepackage[top=35mm, bottom=35mm, left=30mm, right=30mm]{geometry}
\usepackage{mathptmx}
\usepackage{mathrsfs}
\usepackage{verbatim}
\usepackage{url}
\usepackage[all]{xy}
\usepackage{xcolor}
\usepackage[colorlinks=true,citecolor=blue]{hyperref}\newtheorem{theorem}{Theorem}[section]
\newtheorem{lemma}[theorem]{Lemma}
\theoremstyle{definition}
\newtheorem{definition}[theorem]{Definition}

\newtheorem{remark}[theorem]{Remark}

\newtheorem{prop}[theorem]{Proposition}
\newtheorem{ques}[theorem]{Question}
\newtheorem{exam}[theorem]{Example}

\def \RP {{\bf RP}}
\def \Z  {\mathbb{Z}}
\def \N  {\mathbb{N}}

\begin{document}
\title
[Pro-nilfactors and arithmetic progressions]
{Pro-nilfactors of the space of arithmetic progressions in topological dynamical systems}

\author[Z.~Lian]{Zhengxing Lian}
\address[Z.~Lian]{School of Mathmatical Science, Xiamen University, Xiamen City, Fujian Province, 361005, China.}
\email{lianzx@xmu.edu.cn; lianzx@mail.ustc.edu.cn}

\author[J.~Qiu]{Jiahao Qiu}
\address[J.~Qiu]{Wu Wen-Tsun Key Laboratory of Mathematics, USTC, Chinese Academy of Sciences and
School of Mathematics, University of Science and Technology of China,
Hefei, Anhui, 230026, P.R. China}
\email{qiujh@mail.ustc.edu.cn}

\date{\today}

\begin{abstract}
For a topological dynamical system $(X, T)$, $l\in\mathbb{N}$ and $x\in X$,
let $N_l(X)$ and $L_x^l(X)$
be the orbit closures of the diagonal point $(x,x,\ldots,x)$ ($l $ times) under the actions $\mathcal{G}_{l}$ and $\tau_l $ respectively,
where $\mathcal{G}_{l}$ is generated by $T\times T\times \ldots \times T$ ($l $ times)
and $\tau_l=T\times T^2\times \ldots \times T^l$.
In this paper, we show that for a minimal system $(X,T)$ and $l\in \N$,
the maximal $d$-step pro-nilfactor of $(N_l(X),\mathcal{G}_{l})$
is $(N_l(X_d),\mathcal{G}_{l})$,
where $\pi_d:X\to X/\mathbf{RP}^{[d]}=X_d,d\in \mathbb{N}$ is the factor map and $\mathbf{RP}^{[d]}$
is the regionally proximal relation of order $d$.

 Meanwhile, when $(X,T)$ is a minimal nilsystem, we also calculate the pro-nilfactors of $(L_x^l(X),\tau_l)$ for almost every $x$ w.r.t. the Haar measure. In particular, there exists a minimal $2$-step nilsystem $(Y,T)$
 and a countable set $\Omega\subset Y$ such that
  for $y\in Y\backslash \Omega$
the maximal equicontinuous factor of $(L_y^2(Y),\tau_2)$
is not $(L_{\pi_1(y)}^2(Y_{1}),\tau_2)$.
\end{abstract}
\maketitle

\section{Introduction}\

\subsection{Background}\

By a \emph{topological dynamical system} (t.d.s. for short),
we mean a pair $(X,T)$, where $X$ is a compact metric space with a metric $\rho$
and $T:X\to X$ is a homeomorphism.

\medskip

Assume that $(X,T)$ is a minimal system.
For $l\in\N$ and $x\in X$, the orbit closures of $(x,x,\ldots,x)$ under the actions
$\mathcal{G}_l=\langle\sigma_l,\tau_l\rangle$ and $\tau_l$ are denoted by $N_l(X,T,x)$
and $L_x^l(X)$ respectively, where
\[
\tau_l=\tau_l(T)=T\times T^2\times \ldots\times T^l,
\;
\mathrm{and}
\;
\sigma_l=\sigma_l(T)=T\times T\times\ldots\times T.
\]
It is easy to see that $N_l(X,T, x)$ is independent of $x$, which will be denoted by $N_l(X,T)$ or $N_l(X)$.
We call $N_l(X),L_x^l(X)$ \emph{the space of arithmetic progressions of length $l$}
and \emph{the space of simple arithmetic progressions of length $l$ for $x$}
respectively.
A basic result proved by Glasner \cite{GE94} is that $N_l(X)$ is minimal under the
 $\mathcal{G}_l$ action.

Arithmetic progressions in topological dynamical systems relate to the pointwise convergence of multiple ergodic averages. We refer \cite{HSY2014} for more details.

\medskip

In the recent years, the study of the nilsystems
and inverse limits of this kind of dynamics has drawn much interest, since it relates
to many dynamical properties and has important applications in number theory.
We refer to \cite{HK18} and the references therein for a systematic treatment on the subject.

\medskip

 In a pioneer work,
Host-Kra-Maass in \cite{HKM10} introduced the notion of
 {\it regionally proximal relation of order $d$}
for a t.d.s. $(X,T)$, denoted by $\mathbf{RP}^{[d]}$.
For $d\in\N$, we say that a minimal system $(X,T)$ is a \emph{$d$-step pro-nilsystem}
if $\mathbf{RP}^{[d]}=\Delta$ and this is equivalent for $(X,T)$ to be
an inverse limit of $d$-step nilsystems (see \cite[Theorem 2.8]{HKM10}).
For a minimal distal system $(X,T)$, it was proved that
$\mathbf{RP}^{[d]}$ is an equivalence relation and $X/\mathbf{RP}^{[d]}$
is the \emph{maximal $d$-step pro-nilfactor} \cite{HKM10}.
Then Shao-Ye \cite{SY12}
showed that in fact for
any minimal system, $\mathbf{RP}^{[d]}$ is an equivalence relation
and $\mathbf{RP}^{[d]}$
has the so-called lifting property.
Moreover,
$\mathbf{RP}^{[\infty]}=\bigcap_{d=1}^{\infty}\mathbf{RP}^{[d]}$
can be defined and it is also an equivalence relation
for any minimal system \cite{DDMSY13}.

\medskip

 Let $(X,T)$ be a minimal system.
 For $d\in \N\cup \{\infty\}$, let $\pi_d:X\to X/\mathbf{RP}^{[d]}=X_d$ be the factor map.
 In this paper, 
we study the
pro-nilfactors of $(N_l(X),\mathcal{G}_l)$ and $(L_x(X),\tau_l)$ respectively.
In particular, we would like to know whether
the maximal $d$-step pro-nilfactor of the space of
arithmetic progressions of $X$
and the space of
arithmetic progressions of $X_d$
are the same. In \cite{GHSWY20}, the authors studied this question for $d=1$
and they also asked the following questions that

\begin{ques}\cite[Conjecture 2]{GHSWY20}\label{conje}
For $d\in \N$,

(1) Is $N_l(X_d)$  the maximal $d$-step pro-nilfactor of $N_l(X)$?

(2) Is there exists a dense $G_\delta$ set $\Omega$ such that $L_{\pi_d(x)}^l(X_d)$ is the maximal $d$-step pro-nilfactor of
$L_{x}^l(X)$ for $x\in \Omega$?
\end{ques}

In \cite{GHSWY20}, the authors showed that for $d,l\in \mathbb{N}$,

\begin{itemize}
\item  \cite[Theorem 5.6]{GHSWY20} The maximal $d$-step pro-nilfactors of $(N_l(X),\mathcal{G}_l)$ and the one of $(N_l(X_\infty),\mathcal{G}_l)$ coincide;
\item \cite[Theorem 5.7]{GHSWY20} There exists a dense $G_{\delta}$ set $\Omega$ such that for any $x\in \Omega$, the maximal $d$-step pro-nilfactor of $(L_x^{l}(X),\tau_d)$ and the one of $(L_{\pi_{\infty}(x)}^{l}(X),\tau_d)$ coincide.
\end{itemize}

Therefore to study the maximal pro-nilfactors of $(N_l(X),\mathcal{G}_l)$ and $(L_x(X),\tau_l)$, we can restrict to the case that
$X$ is an $\infty$-step pro-nilsystem,
which is an inverse limit of
minimal nilsystems \cite{DDMSY13}.
 Since the inverse limit is easy to handle, we need only to
focus on nilsystems.

\subsection{Arithmetic progressions and pro-nilfactors}\label{arithmetic sec}\

Following the ideas in \cite[Chapter 14]{HK18},
we can view the space of arithmetic progressions of a minimal nilsystem as a nilmanifold.
Thus to answer Question \ref{conje} (1) for nilsystems,
it suffices to compute the pro-nilfactors of a minimal $s$-step nilsystem
$(Z=L/\Gamma,T_1,\ldots,T_k)$.
The similar question was considered in \cite{QZ19} and it was shown that
the maximal $d$-step pro-nilfactor of $Z$ has the form $L/(L_{d+1}\Gamma),d=1,\ldots, s$,
where $L_{i}$ is the $i^{\mathrm{th}}$-step commutator subgroup of $L$, i.e. $L_{i+1}=[L,L_{i}]$ for $i\geq 1$,
and $L_1=L$.

From this, we can give an affirmative answer to Question \ref{conje} (1) for nilsystems.

\begin{theorem}\label{main theorem}
Let $s\geq 2$ be an integer and
let $(X=G/\Gamma,T)$ be a minimal $s$-step nilsystem.
Assume that $G$ is spanned by $G^0$ and the element $t$ of $G$ defining
the transformation $T$.
For $d=1,\ldots,s$, let $X_d=G/(G_{d+1}\Gamma)$.
Then for every $l\in \N$,
the maximal $d$-step pro-nilfactor of $(N_l(X),\mathcal{G}_{l})$
is $(N_l(X_d),\mathcal{G}_{l})$.

\end{theorem}

Thus, combining the previous result in \cite{GHSWY20},
we show

\begin{theorem}\label{main-thm}
Let $(X,T)$ be a minimal system and $d\in \N$.
Then for every $l\in \N$,
the maximal $d$-step pro-nilfactor of $(N_l(X),\mathcal{G}_{l})$
is $(N_l(X_d),\mathcal{G}_{l})$,
where $X_d=X/\mathbf{RP}^{[d]}$.
\end{theorem}

\subsection{Simple arithmetic progressions and pro-nilfactors}\label{simple arithmetic sec}\

Up to now, all of the conclusions are expected. In the process of seeking answers to Question \ref{conje} (2),
we find a surprising result:
there exists a minimal $2$-step nilsystem $(Y,T)$ and
a countable set $\Omega\subset Y$ such that
  for $y\in Y\backslash \Omega$
the maximal equicontinuous factor of $(\overline{\mathcal{O}_{T\times T^2}(y,y)},T\times T^2)$
is not $(\overline{\mathcal{O}_{T\times T^2}(\pi(y),\pi(y))},T\times T^2)$,
where $\pi:Y\to Y/\mathbf{RP}^{[1]}$ is the factor map.
This example gives a negative answer to Question \ref{conje} (2).

Indeed, for a minimal $s$-step nilsystem $(X=G/\Gamma,T)$ we can calculate the pro-nilfactors of the space of simple arithmetic progressions
for $m_X$-a.e. $x$.



Before stating our result, we need define two subgroups of $G^{\Z_+}$ (see also Section \ref{HPG}).
Define 
\[
HP(G)=\{
(gg_1^{\binom{n}{1}}\ldots g_s^{\binom{n}{s}})_{n\in \mathbb{Z_+}},g\in G,g_i\in G_i,i=1,\ldots,s\}
\]
which is called the \emph{Hall-Petresco group},
and
$$
HP_e(G)=\{\phi\in HP(G):\phi(0)=1_G\}.
$$

The groups
$HP(G)$ and $HP_e(G)$ play an important role in the study of arithmetic progressions in nilsystems. We have that

\begin{theorem}\label{commutator of HPe(G)}
Let $s\geq 2$ be an integer and
let $(X=G/\Gamma,T)$ be a minimal $s$-step nilsystem.
Assume that $G$ is spanned by $G^0$ and the element $t$ of $G$ defining
the transformation $T$.
For $m_X$-a.e. $x$ and $d=1,\ldots,s$, the maximal $d$-step pro-nilfactor of $(L_x^l(X),\tau_l)$ is conjugate to the system
$$
(HP_e^{(l)}(G)/\big(HP_e^{(l)}(G)_{d+1}\cdot(HP_e^{(l)}(G)\cap\Gamma^l)\big), \tau_{l,x})
$$
for some nilrotation $\tau_{l,x}$, where $HP_e^{(l)}(G)=\{(\phi(n))_{1\leq n\leq l}:\phi\in  HP_e(G)\}$
and
$HP_e^{(l)}(G)_k$ is the $k^{\mathrm{th}}$-step commutator subgroup of $HP_e^{(l)}(G)$
for $k=1,\ldots,s$.
Moreover, if $G^0$ is simply connected,
then $HP_e^{(l)}(G)_k$ is generated by
$$
\{(g^{n^j})_{1\leq n\leq l}:  g\in G_{j},j=k,\ldots,s \}.
$$
\end{theorem}


\medskip

The paper is organized as follows.
In Section 2, the basic notions used in the paper are introduced.
In Section 3, we give a proof of Theorem \ref{main theorem}
and by using Theorem \ref{main theorem}, we prove Theorem \ref{main-thm}.
In the final section, we show the conclusions in Section \ref{simple arithmetic sec}.

\medskip

\noindent {\bf Acknowledgments.}

The authors would like to thank Professors Wen Huang, Song Shao and Xiangdong Ye for helping discussions.
The second author was supported by NNSF of China (11431012,11971455).

\section{Preliminaries}

In this section we gather definitions and preliminary results that
will be necessary later on.
Let $\Z_+$ ($\N$, $\Z$, respectively) be the set of all non-negative integers (positive integers, integers, respectively).

\subsection{Topological dynamical systems}\label{sec:topological dynamical systems}\

A \emph{transformation} of a compact metric space $X$ is a
homeomorphism of $X$ to itself. A \emph{topological dynamical system}
(t.d.s. for short) is a pair $(X,T)$,
 where $X$ is a compact metric space and $T$
is a transformation of $X$.
For $x\in X,\mathcal{O}_T(x)=\{T^nx: n\in \Z\}$ denotes the \emph{orbit} of $x$.
A t.d.s. $(X,T)$ is called \emph{minimal} if
every point has dense orbit in $X$.

A \emph{homomorphism} between the dynamical systems $(X,T)$ and $(Y,T)$ is a continuous onto map
$\pi:X\to Y$ which intertwines the actions; one says that $(Y,T)$ is a \emph{factor} of $(X,T)$
and that $(X,T)$ is an \emph{extension} of $(Y,T)$. One also refers to $\pi$ as a \emph{factor map} or
an \emph{extension} and one uses the notation $\pi : (X,T) \to (Y,T)$. The systems are said
to be \emph{conjugate} if $\pi$ is a bijection. An extension $\pi$ is determined
by the corresponding closed invariant equivalence relation $R_\pi=\{(x,x')\in X\times X\colon \pi(x)=\pi(x')\}$.

\subsection{Regional proximality of high order}\label{sec:RP}\

For $\vec{n} = (n_1,\ldots, n_d)\in \Z^d$ and $\epsilon=(\epsilon_1,\ldots,\epsilon_d)\in \{0,1\}^d$, we
define
\[
 \vec{n}\cdot \epsilon = \sum_{i=1}^d n_i\epsilon_i .
 \]

\begin{definition}\label{definition of pronilsystem and pronilfactor}
Let $(X,T)$ be a t.d.s. and $d\in \N$.
The \emph{regionally proximal relation of order $d$} is the relation $\textbf{RP}^{[d]} $
(or $\RP^{[d]}(X)$ in case of ambiguity)
defined by: $(x,y)\in\textbf{RP}^{[d]}$ if
and only if for any $\delta>0$, there
exist $x',y'\in X$ and
$\vec{n}\in \Z^d$ such that:
$\rho(x,x')<\delta,\rho(y,y')<\delta$ and
\[
\rho(  T^{\vec{n}\cdot\epsilon} x', T^{\vec{n}\cdot\epsilon}  y')<\delta\;
\text{for all}\; \epsilon\in \{0,1\}^d\backslash\{ \vec{0}\}.
\]
We say a minimal system is a \emph{$d$-step pro-nilsystem}
if its regionally proximal relation of order $d$ is trivial.
\end{definition}

Note that
\[
\ldots\subset \mathbf{RP}^{[d+1]}\subset \mathbf{RP}^{[d]}\subset
\ldots \subset\mathbf{RP}^{[2]}\subset \mathbf{RP}^{[1]}.
\]

\begin{theorem}\cite[Theorem 3.3]{SY12}\label{cube-minimal}
For any minimal system and $d\in \N$,
the regionally proximal relation of order $d$
 is an equivalence relation.
\end{theorem}

The regionally proximal relation of order $d$ allows to construct the \emph{maximal $d$-step pro-nilfactor}
of a minimal system. That is, any factor of $d$-step pro-nilsystem
factorizes through this system.

\begin{theorem}\label{lift-property}\cite[Theorem 3.8]{SY12}
  Let $\pi :(X,T)\to (Y,T)$ be the factor map between minimal systems and $d\in \N$. Then,
  \begin{enumerate}
    \item $(\pi \times \pi) \mathbf{RP}^{[d]}(X)=\mathbf{RP}^{[d]}(Y)$.
    \item $(Y,T)$ is a $d$-step pro-nilsystem if and only if $\mathbf{RP}^{[d]}(X)\subset R_\pi$.
  \end{enumerate}

Particularly, the quotient of $X$ under $\mathbf{RP}^{[d]}(X)$
is the maximal $d$-step pro-nilfactor of $X$.
\end{theorem}
\begin{remark}
  When $d=1$, $\mathbf{RP}^{[1]}$ is nothing but the classical regionally proximal relation.
  For a minimal system $(X,T)$,
  we say it is \emph{equicontinuous} instead of a 1-step pro-nilsystem if its regionally proximal relation
  is trivial and say $X/\mathbf{RP}^{[1]}$ the \emph{maximal equicontinuous factor} 
  instead of the maximal 1-step pro-nilfactor.
\end{remark}

It follows that for any minimal system,
\[
\mathbf{RP}^{[\infty]}=\bigcap_{d\geq1}\mathbf{RP}^{[d]}
\]
is a closed invariant equivalence relation.

Now we formulate the definition of $\infty$-step pro-nilsystems.

\begin{definition}
    A minimal system is an \emph{$\infty$-step pro-nilsystem},
  if the equivalence relation $\mathbf{RP}^{[\infty]}$ is trivial,
  i.e. coincides with the diagonal.
\end{definition}

\subsection{Nilpotent groups, nilmanifolds, nilsystems}\label{sec:nilsystems}\

Let $G$ be a group and denote its
unit element by $1_G$.
For $g,h\in G$, we write $[g,h]=ghg^{-1}h^{-1}$ for the commutator of $g$ and $h$,
we write $[A,B]$ for the subgroup spanned by $\{[a,b]:a\in A,b\in B\}$.
The commutator subgroups $G_j,j\geq1$, are defined inductively by setting $G_1=G$
and $G_{j+1}=[G_j,G]$.
Let $k\geq 1$ be an integer.
We say that $G$ is \emph{k-step nilpotent} if $G_{k+1}$ is the trivial subgroup.

\medskip

Let $G$ be a $k$-step nilpotent Lie group and $\Gamma$ a discrete cocompact subgroup of $G$.
The compact manifold $X=G/\Gamma$ is called a \emph{k-step nilmanifold.}
The group $G$ acts on $X$ by left translations and we write this action as $(g,x)\mapsto gx$.
Let $t\in G$ and $T$ be the transformation $x\mapsto t x$ of $X$.
Then $(X,T)$ is called a \emph{k-step nilsystem}.

We also make use of inverse limits of nilsystems and so we recall the definition of an inverse limit
of systems (restricting ourselves to the case of sequential inverse limits).
If $(X_i,T_i)_{i\in \N}$ are systems with $\text{diam}(X_i)\leq 1$ and $\phi_i:X_{i+1}\to X_i$
are factor maps, the \emph{inverse limit} of the systems is defined to be the compact subset of $\prod_{i\in \N}X_i$
given by $\{(x_i)_{i\in \N}:\phi_i(x_{i+1})=x_i,i\in \N\}$,
which is denoted by $\lim\limits_{\longleftarrow}\{ X_i\}_{i\in \N}$.
It is a compact metric space endowed with the distance $\rho(x,y)=\sum_{i\in \N}1/ 2^i \rho_i(x_i,y_i)$.
We note that the maps $\{T_i\}$ induce a transformation $T$ on the inverse limit.

The following structure theorems characterize inverse limits of nilsystems.

\begin{theorem}[Host-Kra-Maass]\cite[Theorem 1.2]{HKM10}\label{description}
Let $d\geq2$ be an integer.
A minimal system is a $d$-step pro-nilsystem
if and only if
it is an inverse limit of minimal $d$-step nilsystems.
\end{theorem}

\begin{theorem}\cite[Theorem 3.6]{DDMSY13}\label{system-of-order-infi}
A minimal system is an $\infty$-step pro-nilsystem
if and only if it is an inverse limit of minimal nilsystems.
\end{theorem}

The following theorem characterizes the maximal pro-nilfactors of minimal nilsystems.

\begin{theorem}\label{nilfactor}\cite[Theorem 1.2]{QZ19}
Let $s\geq 2$ be an integer and
let $(X=G/\Gamma,T_1,\ldots,T_k)$ be a minimal $s$-step nilsystem.
Assume that $G$ is spanned by $G^0$ and the elements $t_1,\ldots,t_k$ of $G$ defining
the commuting transformations
$T_1,\ldots,T_k$.
For $d=1,\ldots,s$, if $X_d$ is the maximal factor of order $d$ of $X$,
then $X_d$ has the form $G/(G_{d+1}\Gamma)$,
endowed with the translations by the projections
of $t_1,\ldots,t_k $ on $G/G_{d+1}$.
\end{theorem}

\subsection{Hall-Petresco groups}\label{HPG}\

Let $G$ be an $s$-step nilpotent group.
A geometric progression in $G^{\mathbb{Z}_+}$ is defined by the following form
\begin{equation*}\label{de of HP(G)}
(gg_1^{\binom{n}{1}}\ldots g_s^{\binom{n}{s}})_{n\in \mathbb{Z_+}},
\end{equation*}
where $g\in G$ and $g_i\in G_i$ for $i=1,\ldots,s$. The collection of all such progressions is
 {\bf the Hall-Petresco group $HP(G)$} for $G$ (see \cite[Chapter 14]{HK18} that $HP(G)$ is a group). We also define the following group
\begin{equation*}
HP_e(G)=\{\phi\in HP(G):\phi(0)=1_G\}.
\end{equation*}
An observation is that every element $\phi\in HP_e(G)$ has the form
$$\phi(n)=g_1^{\binom{n}{1}}\ldots g_s^{\binom{n}{s}},\quad n\in \Z_+,$$
where $g_i\in G_i$ for $i=1,\ldots,s$.



\section{Arithmetic progressions in topological dynamical systems}

In this section, we will study the space of arithmetic progressions of
a minimal nilsystem.
As an application, we give a proof of Theorem \ref{main theorem}.
Among other things, we can show Theorem \ref{main-thm}.

\subsection{Nilsystems}\label{sec:connect}\

We start by recalling some basic results in nilsystems.
For more details and proofs, see \cite{LA,HK18,PW}.
If $G$ is a nilpotent Lie group,
let $G^0$ denote the connected component of its
unit element $1_G$.
In the sequel, $s\geq2$ is an integer and
$(X=G/\Gamma,T_1,\ldots,T_k)$ is a minimal $s$-step nilsystem with $k$ commuting transformations.
We let $t_1,\ldots,t_k$ denote the elements of $G$ defining the transformations
 $T_1,\ldots,T_k$.

If $(X,T_1,\ldots,T_k)$ is minimal, let $G'$ be the subgroup of $G$
spanned by $G^0$ and $t_1,\ldots,t_k$ and let $\Gamma'=\Gamma\cap G'$,
then we have that $G=G'\Gamma$.
Thus the system $(X,T_1,\ldots,T_k)$ is conjugate to the system
$(X',T_1',\ldots,T_k')$,
where $X'=G'/\Gamma'$ and $T_i'$ is the translation by $t_i$ on $X'$.
Therefore,
without loss of generality, we can restrict to the case that $G$ is spanned by $G^0$
and $t_1,\ldots,t_k$.
We can also assume that $G^0$ is simply connected
(see for example \cite{LA} or \cite{AM} for the case that $G=G^0$
and \cite{AL05} for the general case).
This in turns implies that the commutator subgroups $G_i,i=2,\ldots,s$ are connected and included in $G^0$.
Moreover, $G^0$ is \emph{divisible}, i.e. for any $g\in G^0, d\in \N$, there exists $h\in G^0$
such that $h^d=g$ (see for example \cite[Chapter 10, Corollary 9]{HK18}).

\subsection{Arithmetic progressions in nilmanifolds}\

Let $(X=G/\Gamma,T)$ be a minimal $s$-step nilsystem.
For $d=1,\ldots,s$,
let $HP(G)_d$ be the collection of the element with the following form
\begin{equation*}
(gg_1^{\binom{n}{1}}\ldots g_s^{\binom{n}{s}})_{n\in \mathbb{Z}_+},
\end{equation*}
where $g,g_1,\ldots,g_d\in G_d,g_i\in G_i,i=d+1,\ldots,s$
and $\binom{n}{i}=0$ if $n<i$.

It is clear that $HP(G)_1$ is the Hall-Petresco group $HP(G)$ for $G$ and
\begin{equation}\label{d-step-equal}
HP(G)_d=G_d^{\Z_+}\cap HP(G),
\end{equation}
which implies that every $HP(G)_d$ is also a group.
In \cite[Chapter 15]{HK18}, it was shown that $HP(G)$
is a nilpotent Lie group and its discrete cocompact subgroup is $HP(G)\cap \Gamma^{\Z_+}=\widetilde{\Gamma}$.
Write
\[
HP(X)=HP(G)/\widetilde{\Gamma}.
\]

Define $t^*,t^{\Delta}\in G^{\Z_+}$ as
\[
t^*= 1_G\times t\times t^2\times\ldots,\;
\mathrm{and}\quad
t^{\Delta}= t\times t\times t\times \ldots.
\]
and let $\tau,\sigma$ be
the translations by $t^*$ and $t^{\Delta}$ respectively.

Let $\mathcal{G}$ be the group generated by $\sigma$ and $\tau$.
The nilmanifold $HP(X)$ is invariant under the $\mathcal{G}$ action
as the elements $t^*,t^{\Delta}$ both belong to $HP(G)$.
Moreover, in \cite[Chapter 15]{HK18}, it was shown that the nilsystem $(HP(X),\mathcal{G})$ is minimal.
That is, for any $x\in X$,
\begin{equation}\label{QQQ}
\overline{\mathcal{O}_{\mathcal{G}}(x)}=HP(X).
\end{equation}

We first study the pro-nilfactors of the nilsystem $(HP(X),\mathcal{G})$.
To do this, we need some intermediate lemmas.

\begin{lemma}\label{normal-subgroup}
Let $H$ be a normal subgroup of $G$.
For $g,h\in G$, if $gh\in H$, then $g^nh^n\in H$ for all $n\in \mathbb{Z}_+$.
\end{lemma}

\begin{proof}
For $n\in \mathbb{Z}_+$, write $w_n=g^nh^n$.
In particular, $w_0=1_G,w_1=gh$. 

For positive integer $n$,
as
\[
w_n=g^nh^n=gw_{n-1}h=gh(h^{-1}w_{n-1}h),
\]
and $H$ is normal in $G$,
we deduce inductively that $w_n\in H$.
\end{proof}

\begin{prop}\label{commutator-subgroup}
$HP(G)_d$ is the $d^{\mathrm{th}}$-step
commutator subgroup of $HP(G)$ for $d=1,\ldots,s$.
\end{prop}

\begin{proof}

For $d=1,\ldots,s$,
denote by $\widetilde{HP}(G)_d$ the $d^{\mathrm{th}}$-step
commutator subgroup of $HP(G)$.
In particular, $\widetilde{HP}(G)_1=HP(G)$.
By (\ref{d-step-equal}), the inclusion $\widetilde{HP}(G)_d\subset HP(G)_d$ is trivial.

To prove the converse, we proceed by induction on the degree $s$ of nilpotency.
If $s =1$, then $G_2$ is trivial
and there is nothing to show.
Assume that $s \geq2$, that the result holds for any
$(s-1)$-step nilmanifold, and that $G$ is $s$-step nilpotent.
Write $H = G/G_s$ and let $p: G \to H$
denote the associated quotient map. Then $H$ is an $(s-1)$-step nilpotent
Lie group.
We need the following claims.

\medskip

\noindent {\bf Claim 1:}
$\widetilde{HP}(G)_d\cdot HP(G)_s= HP(G)_d$ for $d=1,\ldots,s$.

\begin{proof}[Proof of Claim 1]
$\widetilde{HP}(G)_d$
and $HP(G)_s$ are obviously subgroups of
$ HP(G)_d$, and therefore
so is $\widetilde{HP}(G)_d\cdot HP(G)_s$.

We next show the converse.
Let $\phi=(\phi(n))_{n\in \Z_+}\in HP(G)_d=G_d^{\Z_+}\cap HP(G)$, then $p\circ\phi=(p\circ \phi(n))_{n\in \Z_+}$
lies in $H_d^{\Z_+}\cap HP(H)=HP(H)_d$.
Thus by the induction hypothesis,
$p\circ\phi$ also belongs to $\widetilde{HP}(H)_d$. It follows that
there exist $\psi\in \widetilde{HP}(G)_d$ and $\theta\in G_s^{\Z_+}$
such that $\phi=\psi \theta$.
Since $HP(G)$ is a group,
$\theta \in HP(G)$.
By (\ref{d-step-equal}), we get that $\theta\in HP(G)_s$.

This shows Claim 1.
\end{proof}

\noindent {\bf Claim 2}:
For $g\in G_d$ and $m=1,\ldots,d$,
the sequence whose terms are
$g^{\binom{n-1}{m}}$ belongs to
$HP(G)_{d}$.

\begin{proof}[Proof of Claim 2]

We show this claim by induction on $m$.

When $m=1$.
For $g\in G_d$, by the definition of the group $HP(G)_d$,
the sequence whose terms are
$g^{\binom{n}{1}}$ and the constant sequence $g^{-1}$
both belong to $HP(G)_{d}$,
and thus the sequence with terms of form
$g^{\binom{n-1}{1}}=g^{\binom{n}{1}}\cdot g^{-1}$ belongs to $HP(G)_d$.

Assume that $1<m \leq d$, that the sequence whose terms are
$h^{\binom{n-1}{m-1}}$ belongs to
$HP(G)_{d}$ for any $h\in G_d$.
Let $g\in G_d$, notice that
$g^{\binom{n-1}{m}}=g^{\binom{n}{m}}\cdot (g^{-1})^{\binom{n-1}{m-1}}$,
and by the definition of the group $HP(G)_d$,
the sequence whose terms are
$g^{\binom{n}{m}}$ belongs to
$HP(G)_{d}$ and
by the induction hypothesis
the sequence whose terms are $(g^{-1})^{\binom{n-1}{m-1}}$
belongs to $HP(G)_d$, and thus the sequence with
terms of the form $g^{\binom{n-1}{m}}$ belongs to $HP(G)_d$.

This shows Claim 2.
\end{proof}

Since $G_s$ is abelian, for $m=0,1,\ldots,s$, the set
\[
H_m=\{a\in G_s:(a^{\binom{n}{m}})_{n\in \Z_+}\in \widetilde{HP}(G)_{s}\}
\]
is a subgroup of $G_s$.

\medskip

\noindent {\bf Claim 3: $H_m=G_s,m=0,1,\ldots,s$.}

\begin{proof}[Proof of Claim 3]

If $m<s$ , then it suffices to show that for $b\in G_{s-1}$ and $c\in G$ , the sequence
whose terms are $[b,c]^{\binom{n}{m}}$ belongs to $\widetilde{HP}(G)_{s}$.
Let $\beta=(b^{\binom{n}{m}})_{n\in \Z_+}$ and $\gamma$ be the constant sequence $c$,
then $\beta\in HP(G)_{s-1}$ and $\gamma\in HP(G)$.
By Claim 1 for $d=s-1$, there exist $\psi\in \widetilde{HP}(G)_{s-1}$ and $\theta\in HP(G)_s$
such that $\beta=\psi \theta$,
and thus $[\psi ,\gamma]\in \widetilde{HP}(G)_{s}$.
As $\theta\in G_s^{\Z_+}$,
we get that
\[
[\beta,\gamma](n)=[\beta(n),\gamma(n)]=
[\psi (n)\theta(n),\gamma(n)]=[\psi (n),\gamma(n)]=[\psi ,\gamma](n)
\]
for any $n\in \Z_+$, which implies
$[\beta ,\gamma]=[\psi ,\gamma]\in \widetilde{HP}(G)_{s}$ and the sequence with
terms of the form $[b^{\binom{n}{m}},c]$ belongs to $\widetilde{HP}(G)_{s}$.
As $G$ is $s$-step nilpotent, the
commutator map $(x,y)\mapsto[x,y]$ taking $G_{s-1}\times  G$ to $G_s$ is a homomorphism in each
coordinate. Thus $[b^{\binom{n}{m}},c]=[b,c]^{\binom{n}{m}}$, and the statement follows.

Assume now that $m=s$. Since $s\geq2$, the group $G_s$ is connected and so is
divisible. Thus it suffices to show that for $a \in G_s$, we have $a^s  \in H_s$ , and thus for all
$b\in  G_{s-1}$ and $c \in G$, the sequence whose terms are
$[b,c]^{s\binom{n}{s}}$ belongs to $\widetilde{HP}(G)_{s}$.
By Claim 2,
the sequence whose terms are
$b^{\binom{n-1}{s-1}}$ belongs to
$HP(G)_{s-1}$,
and the sequence whose terms are $c^n$ belongs to $HP(G)$.
By a similar argument for case $m<s$,
we can get that
the sequence with terms of the
form $[b^{\binom{n-1}{s-1}},c^n]$ belongs to $\widetilde{HP}(G)_s$.
Notice that
$[b,c]^{s\binom{n}{s}}=[b^{\binom{n-1}{s-1}},c^n]$,
thus the statement follows.

This shows Claim 3.
\end{proof}

From Claim 3,
 as $\widetilde{HP}(G)_s$ is a group, $(a_0a_1^{\binom{n}{1}}\ldots a_s^{\binom{n}{s}})_{n\in \mathbb{Z}_+}\in \widetilde{HP}(G)_s$ if $a_0,\ldots,a_s\in H_s=G_s$. Thus we get that
$\widetilde{HP}(G)_s= HP(G)_s$.
By combining Claim 1,
we deduce that
$HP(G)_d$ is the $d^{\mathrm{th}}$-step
commutator subgroup of $HP(G)$ for $d=1,\ldots,s$.
\end{proof}

\begin{lemma}\label{gamma-equal}
$(G_{d}\Gamma)^{\Z_+}\cap HP(G)=HP(G)_{d}\cdot \widetilde{\Gamma}$ for $d=1,\ldots,s$.
\end{lemma}

\begin{proof}
Recall that $\widetilde{\Gamma}=HP(G)\cap \Gamma^{\Z_+}$.
$HP(G)_{d}$
and $\widetilde{\Gamma}$ are obviously subgroups of
$ (G_{d}\Gamma)^{\Z_+}\cap HP(G)$, and therefore
so is $HP(G)_{d}\cdot \widetilde{\Gamma}$.

We next prove the converse.
Let $\phi=(\phi(n))_{n\in \Z_+}\in (G_{d}\Gamma)^{\Z_+}\cap HP(G)$, then
by the definition of the group $HP(G)$
there exist $g\in G,g_m\in G_m,m=1,\ldots,s$
such that for $n\in \Z_+$
\[
\phi(n)=
gg_1^{\binom{n}{1}}\ldots g_s^{\binom{n}{s}}.
\]
As $\phi\in (G_d\Gamma)^{\Z_+}$,
we deduce that $g,g_1,\ldots,g_s\in G_d\Gamma$ inductively.
It suffices to show that
the sequence whose terms are
$g_m^{\binom{n}{m}}$ and the constant sequence $g$ all belong to $HP(G)_{d}\cdot \widetilde{\Gamma}$.
First, as $g\in G_d\Gamma$, there exist $h\in G_d$ and $\gamma\in \Gamma$ such that $g=h\gamma$.
By the definition of the group $HP(G)_d$,
we get that the constant sequence $h$ belongs to $HP(G)_{d}$
and the constant sequence $\gamma$ belongs to $ \widetilde{\Gamma}$,
as was to be shown.

If $m\geq d$,
since $g_m\in G_m$,
we get that the sequence whose terms are
$g_m^{\binom{n}{m}}$ belongs to $HP(G)_{d}$,
and thus belongs to $HP(G)_{d}\cdot \widetilde{\Gamma}$.

If $1\leq m\leq d-1$,
as $g_m\in G_d\Gamma$,
there exist $h_m\in G_d$ and $\gamma_m\in \Gamma$ such that $g_m=h_m\gamma_m$.
Recall that $g_m\in G_m$,
we get $\gamma_m=h_m^{-1}g_m\in G_m$
and thus the sequence whose terms are
$\gamma_m^{\binom{n}{m}}$ belongs to $HP(G)$.

We claim that the sequence whose terms are
$g_m^{\binom{n}{m}}(\gamma_m^{-1})^{\binom{n}{m}}$ belongs to $HP(G)_d$.

To prove the claim, notice that the sequence whose terms are
$g_m^{\binom{n}{m}}(\gamma_m^{-1})^{\binom{n}{m}}$ belongs to $HP(G)$,
it suffices to show that $g_m^{\binom{n}{m}}(\gamma_m^{-1})^{\binom{n}{m}}\in G_d$ for all $n\in \Z_+$.
As $g_m\gamma_m^{-1}=h_m\in G_d$,
by Lemma \ref{normal-subgroup} we have $g_m^k\gamma_m^{-k}\in G_d$ for all $k\in \Z_+$.
In particularly, $g_m^{\binom{n}{m}}(\gamma_m^{-1})^{\binom{n}{m}}\in G_d$ for all $n\in \Z_+$,
and thus the claim follows.
From this claim we deduce that the sequence whose terms are
$g_m^{\binom{n}{m}}$ belongs to $HP(G)_{d}\cdot \widetilde{\Gamma}$.

This completes the proof.
\end{proof}

\begin{lemma}\label{connected}
 The group $HP(G)$ is spanned by $(HP(G))^0$ and the elements $t^*,t^{\Delta}$.
\end{lemma}
\begin{proof}
We first show that the group $HP(G)_2$ is included in $(HP(G))^0$.
Indeed, for every $n\in \Z_+$ the projection $\pi_n:HP(G)_2\to G_2,(\varphi(n))_{n\in \mathbb{Z}_+}\mapsto \varphi(n)$
is surjective and open, and $G_2$ is included in $G^0$ and hence connected,
we get that
the group $HP(G)_2$ is connected and
thus it is included in $(HP(G))^0$.

It is easy to see that any constant sequence
is spanned by $(HP(G))^0$ and the element $t^{\Delta}$.

For $g\in G$, let $\phi_g=(\phi_g(n))_{n\in \Z_+}\in G^{\Z_+}$ such that $\phi_g(n)=g^n$,
then $\phi_g\in HP(G)$.
We claim that there exist $\psi\in (HP(G))^0$ and $k\in \Z$
such that $\phi_g =\psi\cdot (t^*)^{k}$.
As $G$ is spanned by $G^0$ and $t$,
there exist $h\in G^0$ and $k\in \Z$ such that $g=ht^k$.
Since $G^0$ is normal in $G$, by Lemma \ref{normal-subgroup}
we get that $\psi=\phi_g \cdot(t^*)^{-k}\in (G^0)^{\Z_+}$.
As $HP(G)$ is a group and $t^*\in HP(G)$,
$\psi\in HP(G)$.
There exists some $\varphi\in HP(G)_2$
such that $\psi=\phi_h \varphi$.
As $\phi_h\in HP(G^0)\subset (HP(G))^0$,
we deduce that $\psi\in (HP(G))^0$
as was to be shown.

Recall that the group $HP(G)$ is spanned by the constant sequence, the sequence whose terms are
$g^n$,
where $g\in G$
and $HP(G)_2$,
thus the lemma follows.
\end{proof}


Now we calculate the maximal pro-nilfactor of the nilsystem in (\ref{QQQ}) that

\begin{theorem}\label{nil-HP}
Let $(X=G/\Gamma,T)$ be a minimal $s$-step nilsystem.
Assume that $G$ is spanned by $G^0$ and the element $t$ of $G$ defining
the transformation $T$.
For $d=1,\ldots,s$,
let $X_d=G/(G_{d+1}\Gamma)$.
Then the maximal $d$-step pro-nilfactor of $(HP(X),\mathcal{G})$
is $(HP(X_d),\mathcal{G})$.
\end{theorem}

\begin{proof}

Let $X_d=G/(G_{d+1}\Gamma)$ and $p:X\to X_d$ be the quotient map, and let $t'=p(t)$.
Then the transformation induced by $T$ on $X_d$ is the translation by $t'$, which also denoted by $T$.
There is a natural quotient map $p^*:X^{\Z_+}\to X_d^{\Z_+} $
by $(x(n))_{n\in \Z_+} \mapsto (p(x(n)))_{n\in \Z_+} $.
Moreover $p^*$ induces a factor map: $p^*:(HP(X),\mathcal{G})\to (HP(X_d),\mathcal{G})$.

To show the statement, it is sufficient to show
\[
\mathbf{RP}^{[d]}(HP(X))=R_{p^*}.
\]{\color{blue} \footnote{One can see the definition for $\mathbf{RP}^{[d]}$ and $R_{p^*}$ in Sections \ref{sec:topological dynamical systems} and \ref{sec:RP}.}}
As $(HP(X_d),\mathcal{G})$ is a minimal $d$-step nilsystem,
by Theorem \ref{lift-property} we have
\[
\mathbf{RP}^{[d]}(HP(X))\subset R_{p^*}.
\]

We next show the inverse inclusion.
Let $\mathbf{x},\mathbf{y}\in HP(X)$ with $p^{*}(\mathbf{x})=p^{*}( \mathbf{y})$.
Recall that $HP(X)$ is the nilmanifold $HP(G)/\widetilde{\Gamma}$,
where $\widetilde{\Gamma}=HP(G)\cap \Gamma^{\Z_+}$,
then there exists some $\phi\in HP(G)$ such that
$\mathbf{y}=\phi\mathbf{x}$,
which implies that $\phi\in (G_{d+1}\Gamma)^{\Z_+}$.
By Lemma \ref{gamma-equal} we have
\[
\phi\in(G_{d+1}\Gamma)^{\Z_+}\cap HP(G)=HP(G)_{d+1}\cdot \widetilde{\Gamma}.
\]

On the other hand,
by Lemma \ref{connected} and Theorem \ref{nilfactor}, the
maximal $d$-step pro-nilfactor of $HP(X)$ is
\[
HP(G)/(HP(G)_{d+1}\cdot \widetilde{\Gamma}),
\]
which meaning
$(\mathbf{x},\phi\mathbf{x})\in \mathbf{RP}^{[d]}(HP(X))$,
and so $(\mathbf{x},\mathbf{y})\in \mathbf{RP}^{[d]}(HP(X)) $.

We conclude that the maximal $d$-step pro-nilfactor of $(HP(X),\mathcal{G})$
is $(HP(X_d),\mathcal{G})$.
\end{proof}

\subsection{Proof of Theorem \ref{main theorem}}\

Now we are able to give a proof of one of our main results.
\begin{proof}[Proof of Theorem \ref{main theorem}]
Let $X_d=G/(G_{d+1}\Gamma)$ and $p:X\to X_d$ be the quotient map, and let $t'=p(t)$.
Then the transformation induced by $T$ on $X_d$ is the translation by $t'$, which also denoted by $T$.

When $l=1$, the result is trivial, as the system $(N_1(X),\mathcal{G}_1)$
is conjugate to the system $(X,T)$.
For $l\in \N$,
notice that the projection $p_l:X^{\mathbb{Z}_+}\rightarrow X^{l}:(x(n))_{n\in \mathbb{Z}_+}\mapsto
  (x(n))_{0\leq n\leq l}$ induces a factor map
  \[
  p_l: (HP(X),\mathcal{G})
  \rightarrow (N_{l+1}(X),\mathcal{G}_{l+1}).
  \]

By Theorem \ref{nil-HP},
the maximal $d$-step pro-nilfactor of $HP(X)$
is $HP(X_d)$,
thus by Theorem \ref{lift-property}
the maximal $d$-step pro-nilfactor of $N_{l+1}(X)$
is $p_l(HP(X_d))$ which is $N_{l+1}(X_d)$.

This completes the proof.
\end{proof}

\subsection{Proof of Theorem \ref{main-thm}}\

In this subsection, we will show Theorem \ref{main-thm}.
Proving it, we need some intermediate lemmas.
We start from the following simple observation.

\begin{lemma}\label{inverse}
Let $(X,T)$ be an inverse limit of a sequence of minimal systems $\{(X_i,T)\}_{i\in \N}$.
Then for every $l\in \N$,
$(N_l(X),\mathcal{G}_{l})$
is an inverse limit of the sequence
$\{(N_l(X_i),\mathcal{G}_{l})\}_{i\in \N}$.
\end{lemma}

\begin{lemma}\label{nil-factor}\cite[Lemma 5.4]{QZ19}
Let $(X,T)$ be an inverse limit of a sequence of minimal systems $\{(X_i,T)\}_{i\in \N}$.
For $i,d\in \N$, let $Z_{i,d}$ be the maximal $d$-step pro-nilfactor of $X_i$.
Then the maximal $d$-step pro-nilfactor of $X$
is an inverse limit of the sequence $\{(Z_{i,d},T)\}_{i\in \N}$.
\end{lemma}

\begin{lemma}\label{transitive}\cite[Lemma 5.6]{QZ19}
  Let $(X,T)$ be a minimal system and $d\in \N$.
  Let $R\subset X\times X$ be an equivalence relation of $X$ with $R\subset  \mathbf{RP}^{[d]}$,
  then the maximal $d$-step pro-nilfactors of $Y=X/R$ and $X$ coincide.
\end{lemma}

\begin{lemma}\label{con}\cite[Theorem 3.8]{DDMSY13}
Let $(X,T)$ be a minimal system.
If $\mathbf{RP}^{[d]}=\mathbf{RP}^{[d+1]}$ for some $d\in \N$,
then $\mathbf{RP}^{[n]}=\mathbf{RP}^{[d]}$ for all $n>d$.
\end{lemma}

\begin{theorem} \cite[Theorem 5.7]{GHSWY20}\label{coincide-infinity}
Let $(X,T)$ be a minimal system and $d\in\N$.
Then for $l\in \N$
the maximal $d$-step pro-nilfactors
of $N_l(X)$
and $N_l(X_\infty)$ coincide,
where $X_\infty= X/\mathbf{RP}^{[\infty]}$.
\end{theorem}

Now we are able to show Theorem \ref{main-thm}.

\begin{proof}[Proof of Theorem \ref{main-thm}]

Let $X_d=X/\mathbf{RP}^{[d]},d\in \N\cup \{\infty\}$.
It follows from Theorem \ref{coincide-infinity} that
the maximal $d$-step pro-nilfactors of $N_l(X)$ and $N_l(X_\infty)$ coincide.

It suffices to show that
the maximal $d$-step pro-nilfactor of $N_l(X_\infty)$ is $N_l(X_d)$.

If $\mathbf{RP}^{[d]}=\mathbf{RP}^{[d+1]}$,
then $\mathbf{RP}^{[d]}=\mathbf{RP}^{[\infty]}$ by Lemma \ref{con}.
On this moment, $X_\infty$ is equal to $X_d$
and $N_l(X_\infty)$ itself is a $d$-step pro-nilsystem.

If $\mathbf{RP}^{[d]}\neq\mathbf{RP}^{[d+1]}$.
By Theorem \ref{system-of-order-infi},
there exists a sequence of minimal nilsystems $\{ (Y_i ,T)\}_{i\in \N}$
such that $X_\infty=\lim\limits_{\longleftarrow}\{Y_i\}_{i\in \N}$.
Without loss of generality, we may assume that the nilpotency class of
$Y_i$ is not less than $d$ for every $i\in \N$.
Let $X_d$ and $Y_{i,d}$ be the maximal $d$-step pro-nilfactors of
$X$ and $Y_i$ respectively.
By Lemma \ref{transitive}, $X_d$ is also the maximal $d$-step pro-nilfactor of $X_\infty$
and thus $X_d$ is an inverse limit of the sequence $\{Y_{i,d}\}_{i\in \N}$
by Lemma \ref{nil-factor}.
As $Y_i$ is a minimal nilsystem,
the maximal $d$-step pro-nilfactor of $N_l(Y_i)$ is $N_l(Y_{i,d})$
by Theorem \ref{main theorem}.
Note that $N_l(X_\infty)$ is an inverse limit of the sequence
$\{N_l(Y_i)\}_{i\in \N}$ by Lemma \ref{inverse},
we deduce that the maximal $d$-step pro-nilfactor
of $N_l(X_\infty)$ is an inverse limit of the sequence
$\{N_l(Y_{i,d})\}_{i\in \N}$
by Lemma \ref{nil-factor},
which is equal to
$N_l(X_d)$.

We conclude that
the maximal $d$-step pro-nilfactor of $(N_l(X),\mathcal{G}_{l})$
is $(N_l(X_d),\mathcal{G}_{l})$.
\end{proof}

\section{Simple arithmetic progressions in nilsystems}

In the last part of this paper, we first give the example which is mentioned in the introduction. That is,

\subsection{The example}\

\begin{exam}\label{counter}
There is a minimal $2$-step nilsystem $(Y,T)$ and
a countable set $\Omega\subset Y$ such that
  for $y\in Y\backslash \Omega$
the maximal equicontinuous factor of $(\overline{\mathcal{O}_{T\times T^2}(y,y)},T\times T^2)$
is not $(\overline{\mathcal{O}_{T\times T^2}(\pi(y),\pi(y))},T\times T^2)$,
where $\pi:Y\to Y/\mathbf{RP}^{[1]}$ is the factor map.
\end{exam}

Let $G=\Z\times \mathbb{T}\times \mathbb{T}$, with multiplication given by
\[
(k,x,y)*(k',x',y')=(k+k',x+x',y+y'+2kx').
\]
Then $G$ is a Lie group.
Its commutator subgroup $G_2$ is $\{0\}\times \{0\}\times \mathbb{T}$
and $G$ is 2-step nilpotent. The subgroup $\Gamma=\Z\times \{0\}\times \{0\}$
is discrete and cocompact. Let $Y$ denote the nilmanifold $G/\Gamma$ and let $Z=G/(G_2\Gamma)$.
 Let $\alpha$ be irrational, $t=(1,\alpha,\alpha)$
and $T : Y\to Y$ be the translation by $t$. Then $(Y,T)$
is a 2-step nilsystem.
We can view the nilsystem $(Y,T)$ as $T:\mathbb{T}^2\to \mathbb{T}^2,(x,y)\mapsto (x+\alpha,y+2x+\alpha)$,
and $(Z,T_Z)$ as $T_Z:\mathbb{T}\to \mathbb{T},x \mapsto x+\alpha$.

Since $\alpha$ is irrational the rotation $(Z,T_Z)$ is minimal
and $(Y,T)$ is minimal.
By Theorem \ref{nilfactor}, we get that $Z$ is the maximal equicontinuous factor of $Y$.
Let $\pi$ be the factor map, i.e.
$\pi:\mathbb{T}^2\to \mathbb{T},(x,y)\mapsto x$.
For $(x,y)\in \mathbb{T}^2$, $\pi(x,y)=x$ and
\[
\overline{\mathcal{O}_{T_Z\times T_Z^2}(x,x)}=(x,x)+\overline{\{(n\alpha,2n\alpha):n\in \Z\}}
=\{(x+z,x+2z):z\in \mathbb{T}\}.
\]
Thus the system $(\overline{\mathcal{O}_{T_Z\times T_Z^2}(x,x)},T_Z\times T_Z^2)$
is conjugate to the system $(Z,T_Z)$.

\medskip

\noindent {\bf Claim}:
For $(a,b)\in \mathbb{T}^2$,
the maximal equicontinuous factor of
$(\overline{\mathcal{O}_{T\times T^2}(a,b,a,b)},T\times T^2)$
is conjugate to the system $(\overline{\mathcal{O}_{R_{2a}}(0,0)},R_{2a})$,
where $R_{c}:\mathbb{T}^2\to \mathbb{T}^2,(x,y)\mapsto(x+\alpha,y+c)$ for $c\in \mathbb{T}$.
In particular, 
if $\alpha,a$
are rationally independent,
then the system $(\overline{\mathcal{O}_{R_{2a}}(0,0)},R_{2a})$
is not conjugate to the system $(Z,T_Z)$.

\medskip

To show the claim, we start from the following simple observation.

\begin{lemma}\label{conjugation}
Let $(X,T)$ and $(Y,S)$ be minimal systems.
If there exist a continuous onto map $h:X\to Y$ and $x\in X$
such that $h(T^nx)=S^n(h(x))$ for all $n\in \Z$,
then $h$ induces a factor map between systems $(X,T)$ and $(Y,S)$.
\end{lemma}

Now we are in position to show the claim.
\begin{proof}[Proof of Claim]
  For $\beta\in \mathbb{T}$,
let $S_{\beta}:\mathbb{T}^2\to \mathbb{T}^2$ be defined by
\[
S_{\beta}(x,y)= (x+\alpha,y+2x+\alpha+\beta).
\]
If $\beta=0$, then $S_0=T$.
The system $(\mathbb{T}^2, S_\beta)$ is minimal (see for example \cite[Lemma 1.25]{F81}).

Let $U_{\beta}:\mathbb{T}^3\to \mathbb{T}^3$ be defined by
\[
U_{\beta}(x,y,z)=(x+\alpha,y+2x+\alpha+\beta,z+\beta).
\]

\medskip

\noindent {\bf Step 1: A special case.}

Let $h:\mathbb{T}^{4}\to \mathbb{T}^{3}$ be defined by
\[
h(x,y,z,w)
=(x,y,\frac{4y-w}{2}).
\]

Note that
\[
(S_{\beta}\times S_{\beta}^2)^n(0,0,0,0)=(n\alpha,n^2\alpha+n\beta,2n\alpha,4n^2\alpha+2n\beta)
\]
and
\begin{align*}
h((S_{\beta}\times S_{\beta}^2)^n(0,0,0,0)) & =(n\alpha,n^2\alpha+n\beta,n\beta) \\
   & =U^n(0,0,0)\\
   &=U^n(h(0,0,0,0)),
\end{align*}
thus by Lemma \ref{conjugation}, $h$ induces a factor map:
\[
h:(\overline{\mathcal{O}_{S_{\beta}\times S_{\beta}^2}(0,0,0,0)},S_{\beta}\times S_{\beta}^2)\to
(\overline{\mathcal{O}_{U_{\beta}}(0,0,0)},U_{\beta}).
\]
For any $(x_1,x_2,x_3,x_4)\in\overline{\mathcal{O}_{S_{\beta}\times S_{\beta}^2}(0,0,0,0)}$,
we have $x_3=2x_1$.
It follows that $h$ is a bijection and thus $h$ is a conjugation.

Write $L=\overline{\mathcal{O}_{U_{\beta}}(0,0,0)}$.
Notice that for $(x,y_1,z),(x,y_2,z)\in L$ with $y_1\neq y_2$,
then
$((x,y_1,z),(x,y_2,z))\in \mathbf{RP}^{[1]}(L,U_\beta)$,
we deduce that
the maximal equicontinuous factors of
$(L,U_{\beta})$ and $(\overline{\mathcal{O}_{R_{\beta}}(0,0)},R_{\beta})$ coincide.
As the system $(\overline{\mathcal{O}_{R_{\beta}}(0,0)},R_{\beta})$ is equicontinuous,
 it is also the maximal equicontinuous factor of $(L,U_{\beta})$.





Finally,
the maximal equicontinuous factor of $(\overline{\mathcal{O}_{S_{\beta}\times S_{\beta}^2}(0,0,0,0)},S_{\beta}\times S_{\beta}^2)$ is
conjugate to the system
$(\overline{\mathcal{O}_{R_{\beta}}(0,0)},R_{\beta})$.

\medskip

\noindent {\bf Step 2: The general case.}

Fix $(a,b)\in \mathbb{T}^2$.
Let $g:\mathbb{T}^4\to \mathbb{T}^4$ be defined by
\[
g(x,y,z,w)=(x-a,y-b,z-a,w-b).
\]
Note that
\[
(T\times T^2)^n(a,b,a,b)=(a,b,a,b)+
(n\alpha,n^2\alpha+2na,2n\alpha,4n^2\alpha+4na)
\]
and
\begin{align*}
g((T\times T^2)^n(a,b,a,b)) & =(n\alpha,n^2\alpha+2na,2n\alpha,4n^2\alpha+4na) \\
   & =(S_{2a}\times S_{2a}^2)^n(0,0,0,0)\\
   &=(S_{2a}\times S_{2a}^2)^n(g(a,b,a,b)),
\end{align*}
thus by Lemma \ref{conjugation}, $g$ induces a conjugation:
\[
g:(\overline{\mathcal{O}_{T\times T^2}(a,b,a,b)},T\times T^2)
\to(\overline{\mathcal{O}_{S_{2a}\times S_{2a}^2}(0,0,0,0)},S_{2a}\times S_{2a}^2).
\]
Therefore, by Step 1 the maximal equicontinuous factor of
$(\overline{\mathcal{O}_{T\times T^2}(a,b,a,b)},T\times T^2)$
is conjugate to the system $(\overline{\mathcal{O}_{R_{2a}}(0,0)},R_{2a})$.

This completes the proof.
\end{proof}

\subsection{Proof of Theorem \ref{commutator of HPe(G)}}\

Before proving Theorem \ref{commutator of HPe(G)}, we need some lemmas.

\begin{lemma}\label{Dark's Theorem}\cite[Section 3.4]{LA98}
For $l,k_1,\ldots,k_l\in \mathbb{N}$,
let $g_1,\ldots,g_l$ be elements of $G$ where $g_j\in G_{k_j}$, and let $p_1,\ldots,p_l$ be polynomials $\mathbb{Z}^r\rightarrow \mathbb{Z}$ with $\deg p_j\leq k_j$ for $j=1,\ldots,l$. Fix a linear ordering on the set $\mathbb{Z}_+^r$. Then for every $(l_1,\ldots,l_r)\in \mathbb{Z}_+^r$ there exists $z_{l_1,\ldots,l_r}\in G_{l_1+\ldots +l_r}$ such that
\begin{equation}\label{AAA}
\prod_{j=1}^l g_{j}^{p_j(n_1,\ldots,n_r)}=\prod_{I}z_{l_1,\ldots,l_r}^{\binom{n_1}{l_1}\ldots \binom{n_r}{l_r}}
\end{equation}
for all $(n_1,\ldots,n_r)\in \mathbb{Z}_+^r$, where
$I=\{0\leq l_1\leq n_1\}\times \cdots\times \{0\leq l_r\leq n_r\}$
and
 the factors in the product on the right-hand side of (\ref{AAA}) are multiplied in
accordance with the ordering induced on $I$ from $\Z_+^r$ .
\end{lemma}

\begin{lemma}\cite[Chapter 1, Lemma 4]{HK18}\label{nilpotent-group}
  Let $G$ be an $s$-step nilpotent group.
  If $2i+j>s$, then for every $y\in G_j$ the map from $G_i$ to $G_{i+j}$ given by
   $x\mapsto [y,x]$
  is a group homomorphism.
In particular, for any $x_1,\ldots,x_i\in G,y\in G_{s-i}$
and $n_1,\ldots,n_i,n_{i+1}\in \Z$,
  \[
    [\ldots[[x_1^{n_1},x_2^{n_2}],\ldots,x_{i}^{n_i}] ,y^{n_{i+1}}]
    =
     [\ldots[[x_1,x_2],\ldots,x_i], y]^{n_1\cdots n_in_{i+1}}.
    \]

\end{lemma}

\begin{definition}
 Let $G$ be an $s$-step nilpotent and assume.
For $d=1,\ldots,s$, define $A_d\subset G^{\Z_+}$ as the group generated by
\[
\{( g^{n^k})_{n\in \Z_+}:g\in G_k,k=d,\ldots,s\}.
\]
\end{definition}

\begin{prop}\label{d-step-hpe}
Let $G$ be an $s$-step nilpotent and assume that every $G_d$ is divisible for $d=2,\ldots,s$.
Then the $d^{\mathrm{th}}$-step commutator
subgroup of $A_1$ is $A_d$ for $d=2,\ldots,s$.
\end{prop}

\begin{proof}
For $d=2,\ldots,s$,
let $\widetilde{A}_d$ be the $d^{\mathrm{th}}$-step commutator subgroup of $A_1$.
To show the statement,
we need the following claims.
Let $d\geq 2$.
\medskip

\noindent {\bf Claim 1:}
For any $l_1,\ldots,l_d\in \N$ and $z\in G_{l_1+\ldots+l_d}$,
there exist $w_d,\ldots,w_{l_1+\ldots+l_d}\in G_{l_1+\ldots+l_d}$
such that for all $n\in \N$,
\[
z^{\binom{n}{l_1}\ldots \binom{n}{l_d}}
=\prod_{j=d}^{l_1+\ldots+l_d}w_j^{n^j}.
\]
In particular,
 $(z^{\binom{n}{l_1}\ldots \binom{n}{l_d}})_{n\in \Z_+}\in A_{d}$.

\begin{proof}[Proof of Claim 1]
Fix $l_1,\ldots,l_d\in \N$ and let $l=l_1+\ldots+l_d,z\in G_{l}$.
As $d\geq 2$, we have $l\geq d\geq 2$.
Notice that $G_{l}$ is divisible, there exists some $w\in G_{l}$
such that $w^{l_1 !\cdots l_d!}=z$.
Write $l_1 !\cdots l_d!\binom{n}{l_1}\ldots \binom{n}{l_d}=n^l+a_{l-1}n^{l-1}+\ldots+a_dn^d$,
where $a_{l-1},\ldots,a_d\in \Z$.
Then
\[
{z}^{\binom{n}{l_1}\ldots \binom{n}{l_d}}=
w^{l_1 !\cdots l_d!\binom{n}{l_1}\ldots \binom{n}{l_d}}=
w^{n^l+a_{l-1}n^{l-1}+\ldots+a_dn^d}
=w_l^{n^l}w_{l-1}^{n^{l-1}}\cdots w_d^{n^d},
\]
where $w_i=w^{a_i}\in G_{l},i=d,\ldots,l-1$ and $w_l=w$,
as was to be shown.

We next show that $(z^{\binom{n}{l_1}\ldots \binom{n}{l_d}})_{n\in \Z_+}\in A_{d}$.
By the definition of the group $A_d$,
$(w_j^{n^j})_{n\in \Z_+}\in A_d$ for every $j=d,\ldots,l$
and thus the statement follows.
\end{proof}
\medskip

\noindent {\bf Claim 2:}
Let $\phi_1,\phi_2,\ldots ,\phi_d\in A_1$, then for any $(n_1,n_2,\ldots,n_d)\in \N^d$,
\[
 [\ldots[\phi_1(n_1),\phi_2(n_2)],\ldots ,\phi_d(n_d)]
 =\prod_{I}z_{l_1,\ldots,l_d}^{\binom{n_1}{l_1}\ldots \binom{n_d}{l_d}},
 \]
where $z_{l_1,\ldots,l_d}\in G_{l_1+\ldots+l_d}$
and $I=\{(l_1,\ldots,l_d)\in \N^d:l_1+\ldots+l_d\leq s\}$.

In particular, $[\ldots[\phi_1,\phi_2],\ldots ,\phi_d]\in A_d.$

\begin{proof}[Proof of Claim 2]
Let $\phi_1,\phi_2,\ldots ,\phi_d\in A_1$.
It follows from Lemma \ref{Dark's Theorem} that
 \[
\Phi(n_1,n_2,\ldots,n_d)=
 [\ldots[\phi_1(n_1),\phi_2(n_2)],\ldots ,\phi_d(n_d)]
 =\prod_{I}z_{l_1,\ldots,l_d}^{\binom{n_1}{l_1}\ldots \binom{n_d}{l_d}},
\]
where $z_{l_1,\ldots,l_d}\in G_{l_1+\ldots+l_d}$
and $I=\{(l_1,\ldots,l_d)\in \Z_+^d:l_1+\ldots+l_d\leq s\}$.

 We first show that $z_{l_1,\ldots,l_d}=1_G$ if $l_i=0$ for some $1\leq i\leq d$.
 Without loss of generalization, assume that $l_1=0$.
  Notice that $\phi_1(0)=1_G$ and thus
 \[
 1_G= [\ldots[\phi_1(0),\phi_2(n_2)],\ldots \phi_d(n_d)]
 \]
 for all $n_2,\ldots,n_d\in \Z_+$.

 On the other hand, we have
 \begin{equation*}
 \begin{array}{ll}
1_G & =\Phi(0,\ldots,0,0)=z_{0,\ldots,0}\\
     & =\Phi(0,\ldots,0,1)=z_{0,\ldots,0}z_{0,\ldots,0,1}\\
     & =\Phi(0,\ldots,0,2)=z_{0,\ldots,0}z_{0,\ldots,0,1}^2z_{0,\ldots,0,2}\\
&\ldots \ldots
\end{array}
\end{equation*}
which implies that  $z_{0,l_2,\ldots,l_d}=1_G$ for every
$(0,l_2,\ldots,l_d)\in I$, as was to be shown.

Thus by Claim 1, we get that $\Phi \in A_d$.
\end{proof}

It follows from Claim 2 that $\widetilde{A}_d\subset A_d$ for $d=2,\ldots,s$.

\medskip

\noindent {\bf Claim 3:}
For $k\geq d$ and $g\in G_k$,
$(g^{n^d})_{n\in \Z_+}\in \widetilde{A}_d$.

\begin{proof}[Proof of Claim 3]

We first show this claim for $k=s$ and $2\leq d\leq s$.
Let $g\in G_s$ and let $g_1,\ldots,g_{d-1}\in G,g_d\in G_{s+1-d}$ such that
$g=[\ldots[g_1,g_2],\ldots ,g_d]$.
As $(g_i^n)_{n\in \Z_+}\in A_1$ for $i=1,\ldots,d$ and
$g^{n^d}=[\ldots[g_1^n,g_2^n],\ldots, g_d^n]$ for any $n\in \Z_+$
by Lemma \ref{nilpotent-group},
$(g^{n^d})_{n\in \Z_+}\in \widetilde{A}_d$.

We will prove this claim by the decreasing induction for $d$.
When $d=s$, it follows by the argument above for the case $k=s$ and $2\leq d\leq s$.

Let $d<s$ and assume that this statement is true for all $j=d+1,\ldots,s$, i.e.

\noindent {(*)}
for any $j\geq d+1,(z^{n^j})_{n\in \mathbb{Z}_+}\in  \widetilde{A}_j$ for $k\geq j$ and $z\in G_k$.

Now for $d$, we will show that
$(g^{n^d})_{n\in \Z_+}\in \widetilde{A}_d$
for $k\geq d$ and $g\in G_k$ inductively on $k$.
It follows by the argument above for the case $k=s$ and $2\leq d\leq s$
that $(g^{n^d})_{n\in \Z_+}\in \widetilde{A}_d$ for $g\in G_s$.
Let $d\leq k<s$ and assume that

\noindent {(**)}
$(g^{n^d})_{n\in \Z_+}\in \widetilde{A}_d$ for any $j\geq k+1$ and $g\in G_j$.

Let $h\in G_k$ and
let $h_1,\ldots,h_{d-1}\in G,h_d\in G_{k+1-d}$ such that
$h=[\ldots[h_1,h_2],\ldots, h_d]$.
Let $\varphi_i=(h_i^n)_{n\in \Z_+}$ for $i=1,\ldots,d$,
then $\varphi_i\in A_1$ and
$[\ldots[\varphi_1,\varphi_2],\ldots ,\varphi_d]\in \widetilde{A}_d.$

By Claim 2, for any $(n_1,n_2,\ldots,n_d)\in \N^d$,
\[
 [\ldots[\varphi_1(n_1),\varphi_2(n_2)],\ldots ,\varphi_d(n_d)]
 =\prod_{I}z_{l_1,\ldots,l_d}^{\binom{n_1}{l_1}\ldots \binom{n_d}{l_d}},
 \]
where $z_{l_1,\ldots,l_d}\in G_{l_1+\ldots+l_d}$
and $I=\{(l_1,\ldots,l_d)\in \N^d:l_1+\ldots+l_d\leq s\}$.

By taking $n_1=\ldots=n_d=n$,
we have
\[
 [\ldots[\varphi_1(n),\varphi_2(n)],\ldots ,\varphi_d(n)]=
\prod_{I}{z_{l_1,\ldots,l_d}}^{\binom{n}{l_1}\ldots \binom{n}{l_d}}.
\]
In particular, $z_{1,\ldots,1}=h^{n^d}$.

Fix $(l_1,\ldots,l_d)\in \N^d$ with $d+1\leq l\leq s$, where $l=l_1+\ldots+l_d$.
By Claim 1, there exist $w_d,\ldots,w_{l}\in G_{l}\subset G_{d+1}$
such that for all $n\in \N$
\[
z_{l_1,\ldots,l_d}^{\binom{n}{l_1}\ldots \binom{n}{l_d}}
=\prod_{j=d}^{l}w_j^{n^j}.
\]
Therefore by the induction hypothesis (*),
\[
(w_l^{n^l})_{n\in \Z_+}\in \widetilde{A}_l,
(w_{l-1}^{n^{l-1}})_{n\in \Z_+}\in \widetilde{A}_{l-1},\ldots , (w_{d+1}^{n^{d+1}})_{n\in \Z_+}\in \widetilde{A}_{d+1}
\]
and by (**),
$(w_d^{n^d})_{n\in \Z_+}\in \widetilde{A}_d$.
Notice that $\widetilde{A}_d$ is a group and $\widetilde{A}_l\subset \ldots \subset \widetilde{A}_d$, thus
\[
({z_{l_1,\ldots,l_d}}^{\binom{n}{l_1}\ldots \binom{n}{l_d}})_{n\in \Z_+}
=(w_l^{n^l})_{n\in \Z_+} \cdots(w_d^{n^d})_{n\in \Z_+}\in\widetilde{ A}_d.
\]
From this we get that
\[
\big(\prod_{I\backslash\{(1,\ldots,1)\}}{z_{l_1,\ldots,l_d}}^{\binom{n}{l_1}\ldots \binom{n}{l_d}}\big)_{n\in \Z_+}\in \widetilde{ A}_d,
\]
and thus $(h^{n^d})_{n\in \Z_+}\in \widetilde{ A}_d$,
as was to be shown.

This completes the proof.
\end{proof}

Recall that $A_d$ is generated by
\[
\{( g^{n^k})_{n\in \Z_+}:g\in G_k,k=d,\ldots,s\},
\]
thus by Claim 3, $A_d\subset \widetilde{A}_d$.

We conclude that $A_d$ is the $d^{\mathrm{th}}$-step commutator
subgroup of $A_1$ for $d=2,\ldots,s$.
\end{proof}

\begin{theorem}\cite[Chapter 15, Theorem 7]{HK18}\label{minimal-L}
  Let $(X=G/\Gamma,T)$ be a minimal $s$-step nilsystem and for
  $x\in X$, set
  \[
  HP_x(X)=\{\phi\in HP(X):\phi(0)=x\}.
  \]
  Then for $m_X$-almost every $x\in X$, the nilsystem $(HP_x(X),\tau)$ is minimal.
\end{theorem}
Now we are able to show
Theorem \ref{commutator of HPe(G)}.

\begin{proof}[Proof of Theorem \ref{commutator of HPe(G)}]
By Theorem \ref{minimal-L}, there is a full-measure set $\Omega$ such that
$(HP_x(X),\tau)$ is minimal for $x\in \Omega$.
Recall that $HP(G)$ is a nilpotent Lie group, it follows that $HP_e(G)$ is also a nilpotent Lie group.
Write
\[
L(X)=HP_e(G)/\big(HP_e(G)\cap \Gamma^{\Z_+} \big).
\]

For $x\in X$, let $g\in G$ be a lift of $x$ and let $t_x=g^{-1}tg$. 
Define $t_x^*,g^{\Delta}\in G^{\Z_+}$ as
\[
t_x^*= 1_G\times t_x\times t_x^2\times\ldots,\;
\mathrm{and}\quad
g^{\Delta}= g\times g\times g\times \ldots.
\]
and let $\tau_x,\sigma_g$ be
the translations by $t_x^*$ and $g^{\Delta}$ respectively.
Note that $\sigma_g$ is a transformation of $X^{\Z_+}$
and $\tau_x$ is a transformation of $L(X)$ as $t_x^*\in HP_e(G)$.

\medskip

\noindent {\bf Claim 1:}
For any $x\in X$, $\sigma_g$ induces a conjugation:
$\sigma_g:(L(X),\tau_x)\to(HP_x(X),\tau) $.

\begin{proof}[Proof of Claim 1]
Recall that $t_x=g^{-1}tg$,
then $g\cdot t_x^n=t^n\cdot g$ for all $n\in\Z$,
which implies that $\sigma_g \tau_x \phi=\tau \sigma_g \phi$ for any $\phi \in X^{\Z_+}$,
and thus $\sigma_g$ induces a factor map:
$\sigma_g:(L(X),\tau_x)\to(HP_x(X),\tau) $.
Note that $g^{\Delta}HP_e(G)=\{\phi\in HP(G):\phi(0)=g\}$,
and thus $g^{\Delta}$ is homeomorphism betweens $L(X)$ and $HP_x(X)$.

This shows that $\sigma_g:(L(X),\tau_x)\to(HP_x(X),\tau) $ is a conjugation.
\end{proof}

\noindent {\bf Claim 2:}
$HP_e(G)$ is generated by $\big(HP_e(G)\big)^0$ and $\tau_x$.

\begin{proof}[Proof of Claim 2]
Note $t_x t^{-1}=g^{-1}tgt^{-1}\in G^0$ and thus $G$ is spanned by $G^0$ and $t_x$.

We first show that the group $HP_e(G)_2$ is included in $(HP_e(G))^0$.
Indeed, for every $n\in \Z_+$ the projection $\pi_n:HP_e(G)_2\to G_2,(\varphi(n))_{n\in \mathbb{Z}_+}\mapsto \varphi(n)$
is surjective and open, and $G_2$ is included in $G^0$ and hence connected,
we get that
the group $HP_e(G)_2$ is connected and
thus it is included in $(HP_e(G))^0$.

For $g\in G$, let $\phi_g=(\phi_g(n))_{n\in \Z_+}\in G^{\Z_+}$ such that $\phi_g(n)=g^n$,
then $\phi_g\in HP_e(G)$.
We claim that there exist $\psi\in (HP(G)_e)^0$ and $k\in \Z$
such that $\phi_g =\psi\cdot (t_x^*)^{k}$.
As $G$ is spanned by $G^0$ and $t_x$,
there exist $h\in G^0$ and $k\in \Z$ such that $g=ht_x^k$.
Since $G^0$ is normal in $G$, by Lemma \ref{normal-subgroup}
we get $\psi=\phi_g \cdot(t_x^*)^{-k}\in (G^0)^{\Z_+}$.
As $HP_e(G)$ is a group and $t_x^*\in HP_e(G)$,
$\psi\in HP_e(G)$.
There exists some $\varphi\in HP_e(G)_2$
such that $\psi=\phi_h \varphi$.
As $\phi_h\in HP_e(G^0)\subset (HP_e(G))^0$,
we deduce that $\psi\in (HP_e(G))^0$
as was to be shown.

Recall that the group $HP_e(G)$ is spanned by
$\phi_g$ for $g\in G$
and $HP_e(G)_2$,
thus the claim follows.
\end{proof}

By Claim 1, $(HP_x(X),\tau)$ is conjugate to the system $(L(X),\tau_x)$.
It follows form Theorem \ref{nilfactor} and Claim 2 that
the maximal $d$-step pro-nilfactor of $(L(X),\tau_x)$
is
\[
(HP_e(G)/\big(HP_e(G)_{d+1} \cdot(HP_e(G)\cap \Gamma^{\Z_+}) \big),\tau_x).
\]

We next compute the commutator subgroups of $HP_e(G)$.
To do this, we assume that $G^0$ is simply connected.

\medskip

\noindent {\bf Claim 3:}
The $d^{\mathrm{th}}$-step commutator subgroup of
$HP_e(G)$ is generated by
\[
\{( g^{n^k})_{n\in \Z_+}:g\in G_k,k=d,\ldots,s\},
\]
for $d=1,\ldots,s$.

\begin{proof}[Proof of Claim 3]
As $G^0$ is simply connected, every $G_i$ is divisible for $i=2,\ldots,s$.
Thus by Proposition \ref{d-step-hpe}, it suffices to show $HP_e(G)=A_1$.
Recall that $A_1$
is generated by
 \[
\{( g^{n^k})_{n\in \Z_+}:g\in G_k,k=1,\ldots,s\}.
\]

For $d\geq 2$ and $g\in G_d$, there is some $h\in G_d$
such that $h^{d!}=g$.
Write $d!\binom{n}{d}=n^d+a_{d-1}n^{d-1}+\ldots+a_1n$,
where $a_{d-1},\ldots,a_1\in \Z$.
Then
\[
g^{ \binom{n}{d}}=
h^{d!\binom{n}{d}}=
h^{n^d+a_{d-1}n^{d-1}+\ldots+a_1n}
=h_d^{n^d}h_{d-1}^{n^{d-1}}\cdots h_1^{n},
\]
where $h_i=h^{a_i}\in G_d,i=1,\ldots,d-1$ and $h_d=h$.
By the definition of $A_1$, $(h_i^{n^i})_{n\in \Z_+}$ belongs to $A_1$ for every $i=1,\ldots,d$.
This shows that $HP_e(G)\subset A_1$.

On the other hand, for $d\in \N$
there exist $b_0,b_1,\ldots,b_d\in \Z$ such that
$n^d=b_d\binom{n}{d}+\ldots+b_1\binom{n}{1}+b_0$.
Let $n=0$, we get $b_0=0$.
Then for $z\in G_d$,
\[
z^{n^d}=z^{b_d\binom{n}{d}+\ldots+b_1\binom{n}{1}}=z_d^{\binom{n}{d}}\cdots z_1^{\binom{n}{1}},
\]
where $z_i=z^{b_i}\in G_d$ for $i=1,\ldots,d$.
By the definition of $HP_e(G)$, $(z_i^{\binom{n}{i}})_{n\in \Z_+}$ belongs to $HP_e(G)$ for every $i=1,\ldots,d$.
This shows that $ A_1\subset HP_e(G)$.

From this, we deduce that $A_1=HP_e(G)$.
\end{proof}

For $l\in \N$, let $p_l$ be the projection
$p_l:X^{\mathbb{Z}_+}\rightarrow X^{l}:(x(n))_{n\in \mathbb{Z}_+}\mapsto(x(n))_{1\leq n\leq l}$.
For any $n\in \Z$,
we have $p_l(\tau^nx^{\Delta})=\tau_l^n x^l$, where $x^{\Delta}\in X^{\Z_+}$ is the constant sequence $x$ for $x\in X$.
Fix $x\in \Omega$.
Notice that $x^{\Delta}\in HP_x(X)$ and the system $(HP_x(X),\tau)$ is minimal,
it follows from Lemma \ref{conjugation} that $p_l$ induces a factor map
$ p_l: (HP_x(X),\tau)
  \rightarrow (L_x^l(X),\tau_{l}).
$
Moreover, there is a commutative diagram:
$$
\xymatrix{
(L(X),\tau_x) \ar[d]_{p_l} \ar[r]^{\sigma_g} & (HP_x(X),\tau) \ar[d]^{p_l} \\
  (L_{e_X}^l(X),\tau_{l,x}) \ar[r]^{\sigma_{l,g}} & (L_{x}^l(X),\tau_{l})   }
$$
where $\tau_{l,x},\sigma_{l,g}$ are translations by $t_x\times t_x^2\times \ldots\times t_x^l$ and
$g\times \ldots \times g$ ($l$ times) respectively.

Let $HP_e^{(l)}(G)=p_l(HP_e(G))=\{(\phi(n))_{1\leq n\leq l}:\phi\in  HP_e(G)\}$, where $p_l$
is the projection $p_l:G^{\Z_+}\rightarrow G^{l}:(\phi(n))_{n\in \mathbb{Z}_+}\mapsto(\phi(n))_{1\leq n\leq l}$,
then $HP_e^{(l)}(G)$ is a nilpotent Lie group and its discrete subgroup is $HP_e^{(l)}(G)\cap\Gamma^l=\widetilde{\Gamma}^{(l)}$.
Moreover, for $d=1,\ldots,s$
the $d^{\mathrm{th}}$-step commutator subgroup $HP_e^{(l)}(G)_d$ of $HP_e^{(l)}(G)$
 is $p_l(HP_e(G)_d)$ which is generated by
 \[
\{( g^{n^k})_{1\leq n\leq l}:g\in G_k,k=d,\ldots,s\}.
\]

Clearly,
we can view
$L_{e_X}^l(X)$ as the nilmanifold $HP_e^{(l)}(G)/\widetilde{\Gamma}^{(l)}$,
and thus
the maximal $d$-step pro-nilfactor of $(L_x^l(X),\tau_l)$ is conjugate to the system
\[
(HP_e^{(l)}(G)/\big(HP_e^{(l)}(G)_{d+1}\cdot \widetilde{\Gamma}^{(l)}\big), \tau_{l,x}).
\]

This completes the proof.
\end{proof}

\bibliographystyle{amsplain}

\end{document}